\documentclass[12pt,a4paper]{amsart}
\usepackage[english]{babel}
\usepackage[T1]{fontenc}
\usepackage{a4wide,amssymb,amsmath,xypic,ifthen}
\usepackage{graphicx}
\newcommand{\marginnote}[1]{\ifthenelse{\isodd{\thepage}}{\normalmarginpar}
{\reversemarginpar}\marginpar{\fbox{\parbox{15mm}{\sloppy\footnotesize #1}}}}

\newcommand{\qee}{ \hfill\hspace{2pt}$\triangle$}
\newtheorem{thm}{Theorem}[section]
\newtheorem{corol}[thm]{Corollary}
\newtheorem{lemma}[thm]{Lemma}
\newtheorem{prop}[thm]{Proposition}
\newtheorem{defin}[thm]{Definition}

\theoremstyle{remark}
\newtheorem{rema}[thm]{Remark}
 
\newtheorem{exe}[thm]{Example}
 
\newcommand{\cO}{{\mathcal O}}
\newcommand{\PP}{{\mathbb P}}
\newcommand{\Ext}{{\rm Ext}}
\newcommand{\Hom}{{\rm Hom}}

\newcommand{\M}{{\mathfrak M}}
\newcommand{\MMM}{{\mathcal M}}
\newcommand{\NNN}{{\mathcal N}}
\newcommand{\F}{{\mathcal F}}
\newcommand{\E}{{\mathcal E}}
\newcommand{\bG}{{\boldsymbol{\mathcal G}}}
\newcommand{\G}{{\mathcal G}}

\newcommand{\bU}{{\boldsymbol{\mathcal U}}}

\newcommand{\QQ}{{\mathbb Q}}

\newcommand{\C}{{\mathbb C}}
\newcommand{\FF}{{\mathbb F}}

\newcommand{\ZZ}{{\mathbb Z}}

\newcommand{\ch}{\operatorname{ch}}
\newcommand{\id}{\operatorname{id}}
\renewcommand{\max}{{\operatorname{max}}}
\newcommand{\ob}{\operatorname{ob}}
\newcommand{\rk}{\operatorname{rk}}
\newcommand{\END}{{{\mathcal E}nd}}
\newcommand{\EExt}{{{\mathbb E}{\rm{xt}}}}
\newcommand{\HOM}{{{\mathcal H}om}}
\newcommand{\Spec}{{\operatorname{Spec}}}
\newcommand{\tr}{\mbox{\rm tr}\,}
\newcommand{\coe}{{\cO_{E}}}

\newcommand{\bphi}{{\boldsymbol{\phi}}}
\newcommand{\bpsi}{{\boldsymbol{\psi}}}

\newlength{\rrrr}
\newcommand{\isom}[1]{{\settowidth{\rrrr}{$\scriptstyle{x#1x}$}
\xrightarrow{\makebox[\rrrr]{$\scriptstyle{#1}$}}
\hspace{-0.5\rrrr }\hspace{-1.1 em}
\raisebox{- 0.7 ex}{$\sim$}\hspace{0.7\rrrr }
}}
\newcommand{\isoto}{{\longrightarrow\hspace{-1.5 em}
\raisebox{ 0.6 ex}{$\textstyle\sim$}\hspace{0.8 em}}}

\begin{document}
\begin{flushright} SISSA Preprint  06/2009/fm 
 \\  {\tt arXiv:0906.1436 [math.AG]}  
\end{flushright} 
\bigskip\bigskip
\title{Moduli of framed sheaves on projective surfaces}
\bigskip
\date{\today}
\subjclass[2000]{14D20; 14D21;14J60} 
\keywords{Framed sheaves, moduli spaces, stable pairs, instantons}
\thanks{This research was partly supported  by   {\sc prin}    ``Geometria delle variet\`a algebriche''
and by  the European Science Foundation Programme {\sc Misgam.}
 \\[5pt] \indent E-mail: {\tt bruzzo@sissa.it}, {\tt  markushe@math.univ.lille1.fr} }
 \maketitle \thispagestyle{empty}
\begin{center}{\sc Ugo Bruzzo}$^\ddag$ and {\sc Dimitri Markushevich}$^\S$ 
\\[10pt]  {\small 
$^\ddag$ Scuola Internazionale Superiore di Studi Avanzati, \\ Via Beirut 2-4, 34013
Trieste, Italia \\ and Istituto Nazionale di Fisica Nucleare, Sezione di Trieste
}
\\[10pt]  {\small 
\S Math\'ematiques --- B\^at. M2, Universit\'e Lille 1, \\ F-59655 Villeneuve d'Ascq Cedex, France}
\end{center}

\bigskip\bigskip

\begin{abstract} We show that there exists a fine moduli space for torsion-free sheaves on a projective surface, which have a ``good framing" on a   big and nef divisor. This moduli space is a quasi-projective scheme. This is accomplished by showing that such framed sheaves may be considered as stable pairs
in the sense of Huybrechts and Lehn. We characterize the obstruction  to the smoothness of the moduli space, and discuss some examples on rational surfaces. 
  \end{abstract}

\newpage
\section{Introduction}
There has been recently some interest in the moduli spaces of framed sheaves.
One reason is that they are often smooth and provide desingularizations
of the moduli spaces of ideal instantons, which in turn are singular
\cite{NakaBook,NY-I,NY-L}. For this reason, their equivariant cohomology
under suitable toric actions is relevant to the computation of partition functions,
and more generally expectation values of quantum observables in topological
quantum field theory \cite{Nek,BFMT,NY-I,GaLiu,BPT}.
On the other hand, these moduli spaces can be regarded as higher-rank generalizations
of Hilbert schemes of points, and as such they have interesting connections with
integrable systems \cite{KKD,B-ZN}, representation theory \cite{SalaTort}, etc.

While it is widely assumed that such moduli spaces exist and are well behaved,
an explicit analysis, showing that they are quasi-projective schemes and
are fine moduli spaces,
is missing in the literature. In the present paper we provide such a construction
for the case of framed sheaves on smooth projective surfaces, under some mild conditions.
We show that if $D$ is a big and nef curve in a smooth projective surface $X$,
there is a fine quasi-projective moduli space for sheaves that have a ``good framing''
on $D$ (Theorem \ref{main}).   The point here is that the sheaves under consideration are not assumed a priori to be semistable, and the basic idea is to show that there exists a stability condition making all of them stable, so that our moduli space is an open subscheme  of the moduli space of stable pairs in the sense of Huybrechts and Lehn.

In the papers \cite{Nev2,Nev1} T.~Nevins constructed a scheme structure for these
moduli spaces, however we provide a finer analysis, showing that
these schemes are quasi-projective, and in particular are  separated    and of finite type. Moreover we compute the obstruction to the smoothness
of these moduli spaces (Theorem \ref{smoothness}).  In fact,  the tangent space is well known,
but we provide a more precise description   of the obstruction space
than the one given by Lehn \cite{Lehn}. We show that it lies in the kernel of the trace map,
thus extending a previous result of L\"ubke \cite{Lubke} to the non-locally free case.

In some cases there is another way
to give the moduli spaces $\M(r,c,n)$ a structure
of algebraic variety, i.e.,   by using ADHM data. This was done for vector bundles on
$\PP^2$ by Donaldson \cite{Do}, while (always in the locally free case) the case
of the blow-up of $\PP^2$ at a point is studied in A.~King's thesis \cite{King},
and $\PP^2$ blown-up at an arbitrary number of   points was analyzed
by Buchdahl \cite{Bu3}. The general case (i.e., including torsion-free sheaves) is studied
by C.~Rava for Hirzebruch surfaces \cite{Rava2} and A.A.~Henni for multiple blow-ups of
$\PP^2$ at distinct points \cite{Henni}. The equivalence between the two approaches follows from the
fact that in both cases one has \emph{fine} moduli spaces. On the ADHM side, this is shown
by constructing a universal monad on the moduli space 
\cite{OSS,Henni, Rava}.

In the final section we discuss some examples. i.e., framed bundles on Hirzebruch surfaces
with ``minimal invariants", and rank 2 framed bundles on the blowup of $\PP^2$ at one
point.

 In the present article, all the schemes we consider are
separated and are of finite type over $\C$,
and ``a variety'' is a reduced irreducible scheme of finite type over $\C$.
A ``sheaf'' is always coherent, the
term ``(semi)stable'' always means ``$\mu$-(semi)stable'', and the prefix $\mu$- will be omitted.   Framed sheaves are always assumed to be torsion-free.

\bigskip\section{Framed sheaves}

Let us characterize the objects that we shall study. 
\begin{defin} Let $X$ be a smooth projective variety over $\C$, $D\subset X$ an effective divisor, and ${\E_D}$ a sheaf
on $D$. We say that a sheaf $\E$ on $X$ is $(D,{\E_D})$-framable if $\E$ is torsion-free
and there is an isomorphism
$\E_{\vert D} \stackrel{\sim}{\to} {\E_D}$.
An isomorphism $\phi\colon\E_{\vert D} \stackrel{\sim}{\to} {\E_D}$ will be called
a $(D,{\E_D})$-framing of $\E$. 
A framed sheaf is a pair $(\E,\phi)$ consisting of a $(D,{\E_D})$-framable
sheaf $\E$ and a framing $\phi$.
Two framed sheaves $(\E,\phi)$ and $(\E',\phi')$ are isomorphic if there is 
an isomorphism
$f\colon \E\to\E'$ such that $\phi'\circ f_{\vert D}=\phi$.
\end{defin}

Let us remark that our notion of framing is the same as the one used in \cite{Lehn,Nev1,Nev2}, but is more restrictive than that of \cite{HL2},
where a framing is any homomorphism $\E\to {\E_D}$,
not necessarily factoring through an {\em isomorphism}
$\E_{\vert D} \stackrel{\sim}{\to} {\E_D}$.

Our strategy to show that framed sheaves make up ``good'' moduli spaces
will consist in proving that, under some conditions, the pairs $(\E,\phi)$
are stable according to a notion of stability introduced by Huybrechts and Lehn \cite{HL1,HL2}.
 The definition of stability for framed sheaves depends on the choice of a polarization $H$
on $X$ and a positive real number~$\delta$ (in our notation, $\delta$ is
the leading coefficient of the polynomial $\delta$ in the definition of
(semi)stability in \cite{HL2}).

\begin{defin}[\cite{HL1,HL2}] A pair $(\E,\phi)$ consisting of a torsion-free
sheaf $\E$ and its framing $\phi\colon\E_{\vert D} \stackrel{\sim}{\to} {\E_D}$
is said to be $(H,\delta)$-stable, if for any subsheaf $\G\subset\E$
with $0<\rk\G<\rk \E$, the following inequalities hold:
\begin{enumerate} \item 
$\dfrac{c_1(\G)\cdot H}{\rk(\G)}<\dfrac{c_1(\E)\cdot H-\delta}{\rk(\E)}$  when
$\G$ is contained in the kernel of the composition
$\E\rightarrow \E_{\vert D}\isom{\phi}{\E_D}$\medskip;

\item 
$\dfrac{c_1(\G)\cdot H-\delta}{\rk(\G)}<\dfrac{c_1(\E)\cdot H-\delta}{\rk(\E)}$  otherwise.
\end{enumerate}\end{defin}

Remark, that according to this definition, any rank-1 framed sheaf
is $(H,\delta)$-stable for any ample $H$ and any $\delta> 0$.

We shall need the notion of \emph{a family} of such objects. A family of   $(D,{\E_D})$-framed
sheaves on $X$ parametrized by a scheme $S$ of finite type 
is a pair $(\bG,\bphi)$ which satisfies the following conditions:
\begin{enumerate}
\item   $\bG$ is a sheaf on $X\times S$ flat over $S$;
\item  $\bphi$ is a $(D\times S,\operatorname{pr_1}^\ast{\E_D})$-framing for $\bG$. \end{enumerate}

For any sheaf
$\F$ on $X$,\ \ $P^H_\F$ denotes the Hilbert polynomial
$P^H_\F(k)=\chi(\F\otimes\cO_X(kH))$. For a non-torsion sheaf $\F$ on $X$,\ \ $\mu^H$ denotes the slope
of $\F$: \ \ $\mu^H(\F)=\frac{c_1(\F)\cdot H^{n-1}}{\rk\F}$.  

\begin{thm}[\cite{HL1,HL2}] Let $X$ be a smooth projective variety, $H$ an ample divisor
on $X$ and $\delta$ a positive real number. Let $D\subset X$ be an effective divisor,
and ${\E_D}$ a sheaf on~$D$. Then there exists a fine
moduli space $\M=\M_X^H(P)$ of $(H,\delta)$-stable $(D,{\E_D})$-framed sheaves $(\E,\phi)$ with fixed Hilbert polynomial $P=P^H_\E$, and this moduli space is 
a quasi-projective scheme.
\label{goodframing} \end{thm}
Since we are using slope stability, and a more restrictive definition of framing with respect to that of \cite{HL1,HL2},
our moduli space $\M_X^H(P)$  is actually an open subscheme of the moduli space
constructed by Huybrechts and Lehn.

The adjective ``fine'' means the existence of a universal framed sheaf in the following sense: there is a  $(D\times \M,\operatorname{pr_1}^\ast\E_D)$-framed sheaf $(\bU,\bpsi)$ on
$X\times \M$, flat over $\M$, with the property that for every family
$(\bG,\bphi)$
of $(D,\E_D)$-framed sheaves on $X$ parametrized by a scheme of finite type $S$ over
$\C$, there
exist a unique morphism $g\colon S\to \M$ and an isomorphism
of sheaves  ${\boldsymbol\alpha}:\bG\isoto (\operatorname{id}\times g)^\ast \bU$
such that $(\operatorname{id}\times g)^\ast\bpsi\circ
{\boldsymbol\alpha}_{|D\times S}=\bphi$.

Another general result on framed sheaves we shall need is a boundedness theorem
due to M.~Lehn.
Given $X, D, \E_D$ as above, a set $\MMM$ of $(D,\E_D)$-framed pairs $(\E,\phi)$
is bounded is there exists a scheme of finite
type $S$ over $\C$ together with a family $(\bG,\bphi)$ of $(D,\E_D)$-framed pairs
over $S$ such that for any $(\E,\phi)\in \MMM$, there exist $s\in S$ and an isomorphism
$\alpha_s:\G_s\isoto \E$ such that $\phi_s=\phi\circ\alpha_{s|D\times s}$.

\begin{defin}
Let $X$ be a smooth projective variety.
An effective divisor $D$ on $X$ is called a good framing divisor if
we can write $D=\sum n_iD_i$, where $D_i$ are prime divisors
and $n_i>0$, and there exists a nef and big divisor
of the form $\sum a_iD_i$ with $a_i\geq 0$. For a sheaf $\E_D$
on $D$, we shall say that $\E_D$ is a good framing sheaf, if
it is locally free and
there exists a real number $A_0$, $0\leq A_0< \frac1r D^2$, such that for any
locally free subsheaf $\F\subset \E_D$ of constant positive rank, $\frac{1}{\rk\F}\deg c_1(\F) \leq \frac{1}{\rk\E_D}\deg c_1(\E_D)+A_0$.\end{defin}

 \begin{thm}   \label{Lehn-bound} 
Let $X$ be a smooth projective variety of dimension $n\ge 2$, $H$ an ample divisor
on $X$, $D\subset X$ an effective divisor, and ${\E_D}$ a vector bundle  on~$D$.
Assume that $D$ is a good framing divisor. Then for every polynomial $P$ with
coefficients in $\mathbb Q$, the set of
torsion-free shaves  $\E$ on $X$ that satisfy the conditions
$P^H_\E=P$ and $\E_{\vert D}\simeq\E_D$ is bounded.
 \end{thm}
 
 This is proved in \cite{Lehn}, Theorem 3.2.4, for locally free sheaves, but the proof
 goes through also in the torsion-free case, provided that ${\E_D}$ is locally free,
 as we are assuming.

\section{Quasi-projective moduli spaces}

Using the notions introduced in the previous
section, we now can state the main existence result for quasi-projective moduli spaces:

\begin{thm} \label{main}
Let $X$ be a smooth projective surface, $D\subset X$ a big and nef curve, and ${\E_D}$ a good framing sheaf on $D$.  
Then
for any $c\in H^*(X,\QQ)$, there exists an ample divisor $H$ on $X$ and
a real number $\delta>0$ such that all the $(D,\E_D)$-framed sheaves
$\E$ on $X$ with Chern character $\ch(\E)=c$ are $(H,\delta)$-stable, so that
there exists a quasi-projective scheme $\M_X(c)$ which is a fine moduli space for
these framed sheaves.
\end{thm}

\begin{proof}
Let us fix an ample divisor $C$ on $X$. Set $\cO_X(k)=\cO_X(kC)$ and
$\E(k)=\E\otimes \cO_X(k)$ for any sheaf $\E$ on $X$ and for any $k\in\ZZ$.
Recall that the Castelnuovo-Mumford regularity $\rho(\E)$ of a sheaf $\E$ on $X$ is
the minimal integer $m$ such that $h^i(X,\E(m-i))=0$ for all $i>0$.
According to Lehn's Theorem (Theorem \ref{Lehn-bound}), the family $\MMM$ of
all the sheaves $\E$ on $X$ with $\ch(\E)=c$ and $\E_{\vert D}\simeq\E_D$
is bounded. Hence $\rho (\E)$ is uniformly bounded over all $\E\in\MMM$.
By Grothendieck's Lemma (Lemma 1.7.9 in \cite{HLbook}), there exists
$A_1\geq 0$, depending only on $\E_D$, $c$ and $C$, such that $\mu^C(\F)\leq
\mu^C(\E)+A_1$ for all $\E\in\MMM$ and for all nonzero subsheaves $\F\subset\E$.

For $n>0$, denote by $H_n$ the ample divisor $C+nD$. We shall verify that there exists
a positive integer $n$ such that the range
of positive real numbers $\delta$, for which all the framed sheaves $\E$ from $\MMM$
are $(H_n,\delta)$-stable, is nonempty.

Let $\F\subset \E$, $0<r'=\rk\F<r=\rk\E$. Assume first that
$\F\not\subset\ker\big(\E\to\E_{|D}\big)$. Then the $(H_n,\delta)$-stability
condition for $\E$ reads:
\begin{equation}\label{one}
\mu^{H_n}(\F)<\mu^{H_n}(\E)+\left(\frac{1}{r'}-\frac{1}{r}\right)\delta.
\end{equation}
Saturating $\F$, we make $\mu^{H_n}(\F)$ bigger, so we may assume that
$\F$ is a saturated subsheaf of $\E$, and hence that it is locally free.
Then $\F_{|D}\subset \E_{|D}$ and we have:
\begin{equation}\label{two}
\mu^{H_n}(\F)=\frac{n}{r'}\deg c_1(\F_{|D})+\mu^C(\F)\leq \mu^{H_n}(\E)+nA_0+A_1.
\end{equation}

Thus we see that \eqref{two} implies \eqref{one} whenever
\begin{equation}\label{lower}
\frac{rr'}{r-r'}(nA_0+A_1)<\delta .
\end{equation}

Assume now that $\F$ is a saturated, and hence locally free subsheaf of 
$\ker\big(\E\to\E_{|D}\big)\simeq \E(-D)$. Then the
$(H_n,\delta)$-stability
condition for $\E$ is
\begin{equation}\label{three}
\mu^{H_n}(\F)<\mu^{H_n}(\E)-\frac{1}{r}\delta,
\end{equation}
and the inclusion $\F(D)\subset \E$ yields:
\begin{equation}\label{four}
\mu^{H_n}(\F)<\mu^{H_n}(\E)-H_nD+nA_0+A_1=\mu^{H_n}(\E)-(D^2-A_0)n+A_1-DC.
\end{equation}
We see that \eqref{four} implies \eqref{three} whenever
\begin{equation}\label{upper}
\delta< r[(D^2-A_0)n -A_1+DC].
\end{equation}

The inequalities \eqref{lower}, \eqref{upper} for all $r'=1,\ldots,r-1$
have a nonempty interval of common solutions $\delta$ if $$n>
\max\left\{\frac{rA_1-CD}{D^2-rA_0}, 0\right\}.$$
\end{proof}

Remark that up to isomorphism, the quasi-projective structure
making $\M_X(c)$ a fine moduli space is unique, which follows
from the existence of a universal family of framed sheaves over it.

If $D$ is a smooth and irreducible curve and $D^2>0$, then our definition
of a good framing sheaf with $A_0=0$ is just the definition of
semistability. The following is thus an immediate consequence of the theorem:

\begin{corol}
Let $X$ be a smooth projective surface, $D\subset X$ a smooth, irreducible, big and nef curve, and ${\E_D}$ a semistable vector bundle on $D$.  
Then
for any $c\in H^*(X,\QQ)$,
there exists a quasi-projective scheme $\M_X(c)$ which is a fine moduli space of
$(D,\E_D)$-framed sheaves on $X$ with Chern character $c$.
\end{corol}

\section{Infinitesimal study}

Let $X$ be a smooth projective variety, $D$ an effective divisor on $X$,
$\E_D$ a vector bundle on $D$. We shall consider   sheaves $\E$ on $X$
framed to $\E_D$ on $D$. We recall the notion of a simplifying
framing bundle introduced by Lehn.

\begin{defin}
$\E_D$ is simplifying if for any two vector bundles $\E$, $\E'$ on $X$
such that $\E_{|D}\simeq\E'_{|D}\simeq \E_D$, the group
$H^0(X,\HOM(\E,\E')(-D))$ vanishes.
\end{defin}

An easy sufficient condition for $\E_D$ to be simplifying is
$H^0(D, \END (\E_D)\otimes\cO_X(-kD)_{|D})=0$ for all $k> 0$.

Lehn \cite{Lehn} proved that if $D$ is good and $\E_D$ is simplifying,
  there exists a fine moduli space $\M$ of $(D,\E_D)$-framed vector bundles
on $X$ in the category of separated algebraic spaces. Lübke \cite{Lubke}
proved a similar result: if $X$ is a compact complex manifold, 
$D$ a smooth hypersurface (not necessarily
``good'') and if $\E_D$ is simplifying, then the moduli space
$\M$ of $(D,\E_D)$-framed vector bundles exists as a Hausdorff complex space.
In both cases the tangent space $T_{[\E]}\M$ at a point representing the isomorphism
class of a framed bundle $\E$ is naturally identified with
$H^1(X,\END (\E)(-D))$, and the moduli space is smooth at $[\E]$
if $H^2(X,\END (\E)(-D))=0$. Lübke gives a more precise statement
about smoothness: $[\E]$ is a smooth point of $\M$ if
$H^2(X,\END_0 (\E)(-D))=0$, where $\END_0$ denotes the traceless endomorphisms.
Huybrechts and Lehn in \cite{HL1} define the tangent space and give a smoothness criterion
for the moduli space of stable pairs that are more general objects than our framed
sheaves. In this section, we adapt Lübke's criterion to our moduli space $\M_X(c)$,
parametrizing not only vector bundles, but also some non-locally-free
sheaves. When we work with stable framed sheaves, we do not need the assumption that
$\E_D$ is simplifying.

We shall use the notions of the trace map and traceless exts, see Definition 10.1.4
from \cite{HLbook}. Assuming $X$ is a smooth algebraic variety, $\F$ any
(coherent) sheaf on it, and $\NNN$ a locally free sheaf (of finite rank),
the trace map is defined
\begin{equation}
\tr :\Ext^i(\F,\F\otimes\NNN)\to H^i(X,\NNN)\,,\quad  i\in\ZZ,
\end{equation}
and the traceless part of the ext-group, denoted by
$\Ext^i(\F,\F\otimes\NNN)_0$, is the kernel of this map.

We shall need the following property of the trace:

\begin{lemma}\label{traces}
Let $0\xrightarrow{\ }\F\xrightarrow{\alpha }\G\xrightarrow{\beta }\E\xrightarrow{\ } 0$ be an exact triple of sheaves and $\NNN$ a locally free sheaf. Then there are two long exact sequences of ext-functors giving rise to the natural maps
$$
\mu_i:\Ext^i(\F,\E\otimes\NNN)\to\Ext^{i+1}(\E,\E\otimes\NNN)\ ,$$
$$\tau_i: \Ext^i(\F,\E\otimes\NNN)\to\Ext^{i+1}(\F,\F\otimes\NNN)\ ,
$$
and we have $\tr\circ\mu_i=-(1)^i\tr\circ\tau_i$
as maps $\Ext^i(\F,\E\otimes\NNN)\to H^{i+1}(X,\NNN)$.
\end{lemma}

\begin{proof}
This is a particular case of the graded commutativity of the trace with respect to
cup-products on Homs in the
the derived category (see Section V.3.8 in \cite{Ill}):
if $\xi\in\Hom(\F,\E\otimes\NNN[i])$, $\eta \in \Hom(\E,\F[j])$, then $\tr (\xi\circ\eta)
=(-1)^{ij}\tr ((\eta\otimes\id_\NNN)\circ\xi)$. This should be applied to $\xi\in
\Hom(\F,\E\otimes\NNN[i])$ and $\eta=\partial\in \Hom(\E,\F[1])$,
where $\partial$ is the
connecting homomorphism in the distinguished triangle associated to
the given exact triple: 
$$\E[-1]\xrightarrow{-\partial }\F\xrightarrow{\alpha }\G\xrightarrow{\beta }
\E\xrightarrow{\partial }\F[1].$$
\end{proof}

\begin{thm}\label{smoothness}
Let $X$ be a smooth projective surface, $D\subset X$ an effective divisor, ${\E_D}$ a locally free sheaf on $D$, and $c\in H^*(X,\QQ)$ the Chern character of a
$(D,\E_D)$-framed sheaf $\E$ on $X$. Assume that there exists an ample divisor
$H$ on $X$ and a positive real number $\delta$ such that $\E$ is $(H,\delta)$-stable,
and denote by $\M_X(c)$ the moduli space of $(D,\E_D)$-framed sheaves on $X$ with Chern character $c$ which are $(H,\delta)$-stable. Then the tangent space to $\M_X(c)$ is given by
$$
T_{[\E]}\M_X(c)=\Ext^1(\E,\E\otimes\cO_X(-D)),
$$
and $\M_X(c)$ is smooth at $[\E]$ if the traceless ext-group
$$
\Ext^2(\E,\E\otimes\cO_X(-D))_0=\ker
\big[\:\tr : \Ext^2(\E,\E\otimes\cO_X(-D))\to H^2(X,\cO(-D)) \:\big]
$$
vanishes.
\end{thm}

\begin{proof}
We prove this result by a combination 
of arguments of Huybrechts-Lehn and Mukai, so we just give
a sketch, referring  to \cite{HL1,Mukai-symp} for details.
As in Section 4.iv) of \cite{HL1}, the smoothness of $\M=\M_X(c)$ follows
from the $T^1$-lifting property for the complex $\E\to\E_D$. 

Let $A_n=k[t]/(t^{n+1})$, $X_n=X\times \Spec A_n$, $D_n=D\times \Spec A_n$,
$\E_{D_n}=\E_D\boxtimes A_n$,
and let $\E_n\xrightarrow{\alpha_n}\E_{D_n}$ be
an $A_n$-flat lifting of $\E\to\E_D$ to $X_n$. Then the infinitesimal deformations
of $\alpha_n$ over $k[\epsilon]/(\epsilon^2)$ are classified by the hyper-ext
\mbox{$\EExt^1(\E_n, \E_n\xrightarrow{\alpha_n}\E_{D_n})$}\sloppy, and one says
that the $T^1$-lifting property is verified for $\E\to\E_D$ if all the natural
maps
$$
T^1_n:\EExt^1(\E_n, \E_n\xrightarrow{\alpha_n}\E_{D_n})
\to \EExt^1(\E_{n-1}, \E_{n-1}\xrightarrow{\alpha_{n-1}}\E_{D_{n-1}})
$$
are surjective whenever $(\E_n,\alpha_n)\equiv (\E_{n-1},\alpha_{n-1}) \mod (t^n)$.
In loc. cit., the authors remark that there is an obstruction map $\ob$ on
the target of $T^1_n$ which embeds the cokernel
of $T^1_n$ into \mbox{$\EExt^2(\E,\E\to\E_D)$}, so that if the latter vanishes,
the $T^1$-lifting property holds.

In our case, $\E$ is locally free along $D$, so the complex $\E\to\E_D$ is
quasi-isomorphic to $\E(-D)$ and $\EExt^i(\E,\E\to\E_D)=\Ext^i(\E,\E(-D))$.
It remains to prove that the image of $\ob$ is contained in the traceless
part of $\Ext^2(\E,\E(-D))$. This is done by a modification of Mukai's proof
in the non-framed case. 

First we assume that $\E$ is locally free. Then the
elements of $\Ext^1(\E_{n-1},\E_{n-1}(-D_{n-1}))$
can be given by \v Cech 1-cocycles with values
in $\END(\E_{n-1})(-D_{n-1})$ for some open covering
of $X$, and the image of such a 1-cocycle $(a_{ij})$ under the obstruction
map $\Ext^1(\E_{n-1},\E_{n-1}(-D_{n-1}))\to \Ext^2(\E,\E(-D))$ is a \v Cech 2-cocycle $(c_{ijk})$ with values in $\END(\E)(-D)$. A direct calculation shows that
$(\tr c_{ijk})$ is a \v Cech 2-cocycle with values in $\cO_X(-D)$ which is the
obstruction to the lifting of the infinitesimal deformation of the
framed line bundle $\det \E_{n-1}$ from $A_{n-1}$ to $A_n$. As we know that
the moduli space of line bundles, whether framed or not, is smooth, this obstruction
vanishes, so the cocycle $(\tr c_{ijk})$ is cohomologous to 0.

Now consider the case when $\E$ is not locally free. Replacing $\E,\E_D$ by their 
twists $\E(n)$, $\E_D(n)$ for some $n>0$, we may assume that
$H^i(X,\E)=H^i(X,\E(-D))=0$ for $i=1,2$ and that $\E$ is generated by global sections.
Then we get the exact triple of framed sheaves
$$
0\to (\G,\gamma)\to (H^0(X,\E)\otimes\cO_X,\beta)\to (\E,\alpha)\to 0,
$$
where $\G$ is locally free (at this point it is essential that $\dim X=2$
and $X$ is smooth). Then we verify the $T^1$-lifting property for
the exact triples
$$
0\to (\G_n,\gamma_n)\to (\cO_{X_n}^N,\beta_n)\to (\E_n,\alpha_n)\to 0.
$$
The infinitesimal deformations of such exact triples are classified
by $\Hom(\G_n,\E_n(-D_n))$, and the obstructions lie in
$\Ext^1(\G,\E(-D))$. We have two connecting homomorphisms
$\mu_1:\Ext^1(\G,\E(-D)\to \Ext^2(\E,\E(-D))$ and
$\tau_1:\Ext^1(\G,\E(-D)\to \Ext^2(\G,\G(-D))$.
Our hypotheses on $\E$ imply that: 1) every infinitesimal deformation
of $(\E_n,\alpha_n)$ lifts to that of the triple, and 2) $\mu_1$ is an isomorphism,
that is, the infinitesimal deformation of $\E_n$ is unobstructed
if and only if that of the triple is. 
By Lemma \ref{traces}, $\tr (\mu_1(\xi))=-\tr (\tau_1(\xi))$
in $H^2(X,\cO_X(-D))$. As in 1.10 of \cite{Mukai-symp}, $\tau_1(\xi)$
is the obstruction $\ob (G_{n-1},\gamma_{n-1})$ to lifting $(G_{n-1},\gamma_{n-1})$
from $A_{n-1}$ to $A_n$. As $G_{n-1}$ is locally free, we can use the
\v Cech cocycles as above and see that $\tr(\tau_1(\xi))\in H^2(X,\cO_X(-D))$ is the obstruction to lifting $(\det G_{n-1}, \det \gamma_{n-1})$, hence it is zero
and we are done.
\end{proof}

The following Corollary describes a situation where the moduli space 
 $\M_X(c)$ is smooth (hence, every connected component is a smooth  quasi-projective variety).
 
 \begin{corol}In addition to the hypothesis of Theorem \ref{smoothness}, let us assume
that $D$ is irreducible, that $(K_X+D)\cdot D<0$, and choose 
 a trivial   bundle   as framing bundle. Then the moduli space
 $\M_X(c)$ is smooth.
 \end{corol}
 
 This happens for instance when $X$ is a Hirzebruch surface, or the blow-up of $\PP^2$
at a number of distinct points, taking for  $D$     the inverse
image of a generic line in $\PP^2$ via the birational morphism
$X\to\PP^2$. In this case one can also compute the dimension of
the moduli space, obtaining $\dim \M_X(c)=2rn$,
with $r=\rk(\E)$ and
$$c_2(\E)-\frac{r-1}{2r}c_1(\E)^2 = n\,\varpi ,$$
where $\varpi$ is the fundamental class of $X$. When $X$ is the
$p$-th Hirzebruch surface $\FF_p$ we shall denote this moduli space by
$\M^p(r,k,n)$ if $c_1(\E)=kC$, where $C$ is the unique curve
in $\FF_p$ having negative self-intersection.

 The next example shows that the moduli space may be nonsingular even if
 the group $\Ext^2(\E,\E\otimes\cO_X(-D))$ does not vanish.

\begin{exe} For $r=1$ the moduli space $\M(1,0,n)$ is isomorphic to the Hilbert scheme $X_0^{[n]}$ parametrizing length $n$ 0-cycles in $X_0=X\setminus D$. Of course this
space is a smooth quasi-projective variety of dimension $2n$. Indeed in this case
the trace morphism $\Ext^2(\E,\E\otimes\cO_X(-D))\to H^2(X,\cO(-D))$  is an isomorphism.
\end{exe}

\section{Examples}
\subsection{Bundles with small invariants on Hirzebruch surfaces} \label{grass}
Let  $X$ be   the $p$-th Hirzebruch surface $\FF_p$, and normalize 
the Chern character so that $0 \le k \le r-1$.
 It has been shown in \cite{BPT} that the  moduli space $\M^p(r,k,n)$
 is nonempty if and only if the bound
$$n \ge N =\frac{pk}{2r}(r-k)$$
is satisfied. The moduli spaces $\M^p(r,k,N)$
can be explicitly characterized: $\M^p(r,k,N)$
is a rank $k(r-k)(p-1)$ vector bundle on the Grassmannian
$G(k,r)$ of $k$-planes in $\C^r$ \cite{Rava}; in particular, $\M^1(r,k,N)\simeq G(k,r)$,
and $\M^2(r,k,N)$ is isomorphic to the tangent bundle of $G(k,r)$.
This is consistent with instanton counting, which  shows  that the spaces $\M^p(r,k,N)$ have the same Betti numbers as $G(k,r)$ \cite{BPT}.

\subsection{Rank 2 vector bundles on $\FF_1$}
We study in some detail the moduli spaces\break $\M^1(2,k,n)$.
As the analyses in \cite{Soro} and \cite{Tikho}  show,
the non-locally free case turns out to be very complicated as soon as the value
of $n$ exceeds the rank. So we consider only locally free sheaves.
To simplify notation we call this moduli space $\hat M(k,n)$, where $n$ denotes now
the second Chern class. We normalize
$k$ so that it will assume only the values 0 and $-1$.
Moreover we shall denote by $M(n)$ the moduli space of rank 2 bundles on $\PP^2$, with second Chern class $n$, that are framed on the ``line at infinity'' $\ell_\infty\subset \PP^2$ (which we identify
with the image of $D$ via the blow-down morphism $\pi\colon\FF_1\to\PP^2$).

Let us start with the case $k=-1$. We introduce a stratification on $\hat M(-1,n)$
according to the splitting type of the bundles it parametrizes on the exceptional
line $E\subset \FF_1$
$$ \hat M(-1,n) = Z_0(-1,n)  \supset Z_1(-1,n) \supset  Z_2(-1,n) \supset \dots $$
defined as follows: if $Z_k^0(-1,n) = Z_k(-1,n) \setminus Z_{k+1}(-1,n) $ then
 $$Z_k^0(-1,n)=\{\E\in\hat M(-1,n)\,\vert\,\E_{\vert E }\simeq \coe(-k)\oplus\coe(k+1)\}\,.$$
 \begin{prop} There is a map $$ F_1 \colon \hat M(-1,n)\to\coprod_{k=0}^nM(n-k)$$
 which restricted to the subset $Z_k^0(-1,n) $ yields a morphism
 $$Z_k^0(-1,n) \to M(n-k)$$
 whose fibre is an  open set in $\Hom(\sigma^\ast\E_{\vert E},\coe(k))/\C^\ast\simeq\PP^{2k+1}$,
 made by $k$-linear forms that have no common zeroes
on the exceptional line. \end{prop}
\begin{proof} 
We start by considering $ Z_0^0(-1,n)$. The morphism
$ Z_0^0(-1,n)\to M(n)$ is given by $\E_1\mapsto \E= (\pi_\ast \E)^{\ast\ast}$. The fibre of
this morphism includes a $\PP^1$. To show that this is indeed a $\PP^1$-fibration we need to check that $ \E_1$ has no other
deformations than those coming from the choice of a point in $M(n)$ and a point in this
$\PP^1$. This follows  from the   equalities
$$\dim\Ext^1(\E_1,\E_1(-E)) =  \dim\Ext^1(\E,\E(-\ell_\infty) + 1 $$
$$\Ext^2(\E_1,\E_1(-E ))=0$$
Note that this result is compatible with the isomorphism $\M^1(r,k,N)\simeq G(k,r)$ 
mentioned in Section \ref{grass}.

In general, if $\E_1\in Z_k^0(-1,n)$ with $k\ge 1$, so that 
${\E_1}_{\vert E}\simeq\coe(k+1)\oplus\coe(-k)$, the direct image $\pi_\ast(\E_1(kE))$ is locally free. This defines the  morphism  
$Z_k^0(-1,n) \to M(n-k)$.
\end{proof}

We consider now the case $k=0$. One has $Z_0^0(0,n)\simeq M(n)$. We study the other strata by reducing to the odd case.  f $\E_1\in Z_k^0(0,n)$, there is a unique surjection $\alpha\colon\E_1\to\coe(-k)$; let $\F$ be the kernel. Restricting $ 0 \to \F \to \E_1 \to \coe(-k) \to 0$  we get an exact sequence
$$ 0 \to \coe(1-k)\to \F_{\vert E} \to  \coe(k) \to 0 $$
so that
$$ \F_{\vert E} \simeq \coe(a+1)\oplus\coe(-a) \qquad\text{with}\qquad -k\le a\le k -1.$$
A detailed analysis shows that $a=k-1$. 
As a result we have:
\begin{prop} For all $k\ge 1$ there is a morphism
$$Z_k^0(0,n)\to M(n-2k+1)$$
whose fibres have dimension $2k-1$. 
\end{prop}


\bigskip\frenchspacing
\begin{thebibliography}{10}

\bibitem{B-ZN}
{\sc D.~Ben-Zvi and T.~Nevins}, {\em Flows of {C}alogero-{M}oser systems}, Int.
  Math. Res. Not. IMRN,  (2007), pp.~Art. ID rnm105, 38.

\bibitem{BFMT}
{\sc U.~Bruzzo, F.~Fucito, J.~F. Morales, and A.~Tanzini}, {\em Multi-instanton
  calculus and equivariant cohomology}, J. High Energy Phys.,  (2003), pp.~054,
  24 pp. (electronic).

\bibitem{BPT}
{\sc U.~Bruzzo, R.~Poghossian, and A.~Tanzini}, {\em Poincar\\e polynomial
of moduli spaces of framed sheaves on (stacky)  {H}irzebruch surfaces}.
\newblock {\tt arXiv:0909.1458 [math.AG]}.

\bibitem{Bu3}
{\sc N.~P. Buchdahl}, {\em Blowups and gauge fields}, Pacific J. Math., 196
  (2000), pp.~69--111.

\bibitem{Do}
{\sc S.~K. Donaldson}, {\em Instantons and geometric invariant theory}, Comm.
  Math. Phys., 93 (1984), pp.~453--460.

\bibitem{GaLiu}
{\sc E.~Gasparim and C.-C.~M. Liu}, {\em The {N}ekrasov conjecture for toric
  surfaces}.
\newblock {\tt arXiv:0808.0884 [math.AG]}.

\bibitem{Henni}
{\sc A.~A. Henni}, {\em Monads for torsion-free sheaves on multi-blowups of the
  projective plane}.
\newblock {\tt arXiv:0903.3190 [math.AG]}.

\bibitem{HL1}
{\sc D.~Huybrechts and M.~Lehn}, {\em Framed modules and their moduli},
  Internat. J. Math., 6 (1995), pp.~297--324.

\bibitem{HL2}
\leavevmode\vrule height 2pt depth -1.6pt width 23pt, {\em Stable pairs on
  curves and surfaces}, J. Alg. Geom., 4 (1995), pp.~67--104.

\bibitem{HLbook}
\leavevmode\vrule height 2pt depth -1.6pt width 23pt, {\em The geometry of
  moduli spaces of sheaves}, Aspects of Mathematics, E31, Friedr. Vieweg \&
  Sohn, Braunschweig, 1997.

\bibitem{Ill}
{\sc L.~Illusie}, {\em Complexe cotangent et déformations. I}, vol.~239 of
  Lecture Notes in Mathematics, Springer-Verlag, Berlin-New York, 1971.

\bibitem{KKD}
{\sc A.~Kapustin, A.~Kuznetsov, and D.~Orlov}, {\em Noncommutative instantons
  and twistor transform}, Comm. Math. Phys., 221 (2001), pp.~385--432.

\bibitem{King}
{\sc A.~King}, {\em Instantons and holomorphic bundles on the blown-up plane}.
\newblock Thesis, Oxford University, 1989.

\bibitem{Lehn}
{\sc M.~Lehn}, {\em Modulr\"aume gerahmter {V}ektorb\"undel}.
\newblock Ph.D. thesis, Bonn 1992.

\bibitem{Lubke}
{\sc M.~L\"ubke}, {\em The analytic moduli space of framed vector bundles}, J.
  Reine Angew. Math., 441 (1993), pp.~45--59.

\bibitem{Mukai-symp}
{\sc S.~Mukai}, {\em Symplectic structure of the moduli space of sheaves on an
  abelian or {$K3$} surface}, Invent. Math., 77 (1984), pp.~101--116.

\bibitem{NakaBook}
{\sc H.~Nakajima}, {\em Lectures on {H}ilbert schemes of points on surfaces},
  vol.~18 of University Lecture Series, American Mathematical Society,
  Providence, RI, 1999.

\bibitem{NY-L}
{\sc H.~Nakajima and K.~Yoshioka}, {\em Lectures on instanton counting}, in
  Algebraic structures and moduli spaces, vol.~38 of CRM Proc. Lecture Notes,
  Amer. Math. Soc., Providence, RI, 2004, pp.~31--101.

\bibitem{NY-I}
\leavevmode\vrule height 2pt depth -1.6pt width 23pt, {\em Instanton counting
  on blowup. {I}. 4-dimensional pure gauge theory}, Invent. Math., 162 (2005),
  pp.~313--355.

\bibitem{Nek}
{\sc N.~A. Nekrasov}, {\em Seiberg-{W}itten prepotential from instanton
  counting}, Adv. Theor. Math. Phys., 7 (2003), pp.~831--864.

\bibitem{Nev2}
{\sc T.~A. Nevins}, {\em {Moduli spaces of framed sheaves on certain ruled
  surfaces over elliptic curves}}, Int. J. Math., 13 (2002), pp.~1117--1151.

\bibitem{Nev1}
\leavevmode\vrule height 2pt depth -1.6pt width 23pt, {\em {Representability
  for some moduli stacks of framed sheaves}}, Manuscr. Math., 109 (2002),
  pp.~85--91.

\bibitem{OSS}
{\sc C.~Okonek, M.~Schneider, and H.~Spindler}, {\em Vector bundles on complex
  projective spaces}, vol.~3 of Progress in Mathematics, Birkh\"auser Boston,
  Mass., 1980.

\bibitem{Rava2}
{\sc C.~L.~S. Rava}.
\newblock Ph.D. thesis, SISSA (Trieste). Work in progress.

\bibitem{Rava}
\leavevmode\vrule height 2pt depth -1.6pt width 23pt, {\em {ADHM} data for
  framed sheaves on {H}irzebruch surfaces}.
\newblock {I}n preparation, 2009.

\bibitem{SalaTort}
{\sc F.~Sala and P.~Tortella}, {\em Representations of the {H}eisenberg algebra
  and moduli spaces of framed sheaves}.
\newblock In preparation, 2009.

\bibitem{Soro}
{\sc M.~E. Sorokina}, {\em Birational properties of moduli spaces of rank 2
  semistable sheaves on the projective plane}.
\newblock Ph.D. thesis (in Russian), Yaroslavl, 2006.

\bibitem{Tikho}
{\sc A.~S. Tikhomirov}, {\em On birational transformations of {H}ilbert schemes
  of an algebraic surface}, Mat. Zametki, 73 (2003), pp.~281--294.

\end{thebibliography}
\end{document}